\documentclass[10pt]{article}

\usepackage{amssymb,latexsym,amsmath,mathrsfs}
\usepackage{amsfonts}
\usepackage{amscd}
\usepackage{hyperref}
\usepackage{dsfont}
\usepackage{bbm}
\usepackage{xcolor}
\usepackage{rotating}
\usepackage{newlfont}
\usepackage{enumitem}
\usepackage[numbers,sort&compress]{natbib}

\date{}
\makeatother

\def\h25{\hspace{-.3cm}}

\def\leq{\leqslant}

\newtheorem{definition}{Definition}[section]
\newtheorem{proposition}{Proposition}[section]
\newtheorem{theorem}{Theorem}[section]

\newtheorem{lemma}{Lemma}[section]
\newtheorem{remark}{Remark}[section]

\newenvironment{proof}{\noindent {\bf
Proof.}}{\hfill\rule{3mm}{3mm}\par\medskip}

\begin{document}

%\baselineskip=18
\title{New notion of mild solutions 
%together with its existence and uniqueness analysis
%and its existence and uniqueness results
for nonlinear differential systems involving Riemann-Liouville derivatives of higher order with non-instantaneous impulses}

%\title{Existence 
%%%and uniqueness results 
%for mild solutions to higher order Riemann-Liouville nonlinear
%%%fractional 
%differential systems with non-instantaneous impulses via fractional resolvent}
%The pth moment asymptotical ultimate boundedness of pantograph stochastic differential equations with time-varying coefficients
%New notion of mild solutions together with its existence and uniqueness analysis for higher order Riemann-Liouville nonlinear differential systems involving non-instantaneous impulses

%Total Approximate Controllability of Higher Order Riemann-Liouville Fractional Differential Systems with Non-instantaneous Impulses 
%Higher Order Non-instantaneous Impulsive Riemann-Liouville Fractional Differential Systems: Existence and Controllability Analysis
%}
\author
{Lavina Sahijwani
\thanks{Indian Institute of Technology Roorkee--247667, India. Email: sahijwani.lavina@gmail.com}
% Research is partially supported by grants NSF DMS1600778 and NASA MIRO NX15AQ06A.}
\and
N. Sukavanam
\thanks{Indian Institute of Technology Roorkee--247667, India. Email: n.sukavanam@ma.iitr.ac.in}
}
\maketitle

\begin{abstract}
The artefact is dedicated towards the inspection of nonlinear fractional differential systems involving Riemann-Liouville  derivative with higher order and fixed lower limit including non-instantaneous impulses for existence and uniqueness results in Banach spaces. Motive of the paper is to set sufficient conditions to guarantee existence of mild solution in Banach spaces. Firstly, appropriate integral type initial conditions depending on the impulsive functions are chosen at suitable points. Mild solution of the concerned system is constructed using fractional resolvent. Subsequently, existence and uniqueness results are established  under sufficient assumptions utilising fixed point approach. An example is presented at the end to validate the methodology proposed.

%Lastly, an application is illustrated to validate the theory  presented. 

%are chosen in a different way depending on the impulsive functions at suitable points.
%where impulses ends. 
%The chosen system of inspection is non-linear fractional differential systems involving Riemann-Liouville higher order derivative along with non-instantaneous impulses. The initial conditions are chosen in a different way depending on the impulsive functions.
%First in order, is the construction of mild solution of the concerned system followed by proving the existence and uniqueness results utilising fixed point approach.
%This artefact introduces the concept of total approximate controllability in order to seek approximate controllability of the concerned system at points other than the final point also. The main motivation behind introducing the concept is to carefully examine the behaviour of dynamical systems affected by impulses of non-instantaneous kind. The chosen system of inspection is non-linear fractional differential systems involving Riemann-Liouville higher order derivative along with non-instantaneous impulses. The initial conditions are chosen in a different way depending on the impulsive functions. First in order, is the construction of mild solution of the concerned system followed by proving the existence and uniqueness results utilising fixed point approach. Next, is to establish the total approximate controllability for the system under suitable assumptions using fractional resolvent, appropriately defined Nemytskii operators  and iterative technique.

\bigskip


\end{abstract}

\section{Introduction} \label{1}
Riemann-Liouville and Caputo type derivatives have been the focus of attention for many analysts in the field of fractional calculus (see \cite{rp ag1,rp ag main,book hindawi,A shukla1,abdul,zhou1,bala nonlinear,A shukla2,kilbas:srivast,liu:siam,pazy,podlubny,jde1,lavina,A shukla3,min zhang,zhou2,zhou3,zhou} and references therein). In the sense\textcolor{blue}{,} that it allows the function to endure discontinuity at origin, the Riemann-Liouville derivative triumphs over Caputo. On the other hand, it does not allow the use of standard initial conditions; in the Riemann-Liouville situation, the initial conditions are either integral or weighted initial conditions. The manifestation of physical significance to the beginning conditions used in Riemann-Liouville fractional order viscoelastic systems was credited to Heymans and Podlubny \cite{hey:podlubny}.\\
%%In domain of fractional calculus, Riemann-Liouville and Caputo type derivatives have maintained to be the centre of attention for numerous analysts. Riemann-Liouville derivative shows supremacy over Caputo in the sense that it allows the function involved to bear discontinuity at origin. On the other hand, it  doesn't allow the use of traditional initial conditions, the initial conditions involved in Riemann-Liouville case are either in the integral form or are weighted initial conditions. Heymans and Podlubny \cite{hey:podlubny} were the ones accredited for the manifestation of physical significance to the initial conditions used in 
%%%case 
%%regard of Riemann-Liouville fractional order viscoelastic systems.\\
\indent Recently, researchers worldwide gravitate towards the analysis of impulsive evolution systems with the focus on suitable mathematical modelling, existence of integral solution, its stability, controllability and much more. In realistic modelling,
% In real life mathematical models, 
the occurrence of perturbations, due to extrinsic intercessions, 
is inevitable, yet, unpredictable. These perturbations or sudden changes are nothing\textcolor{blue}{,} but impulses affecting the solution's behaviour majorly. The literature consists of artefacts
% mainly 
addressing mainly two types of impulses, one is, instantaneous impulses which occur suddenly and has its effect for a very smaller period of time, and the other is, non-instantaneous impulses, which start all at once but hold for a finite time interval, for example, injection of a medical drug into the human body is sudden but takes time to stabilize its effect. Mathematically, instantaneous impulses are represented by point-wise break in the dynamical systems whereas non-instantaneous impulses are represented by interval type break. Basically, non-instantaneous impulses are nothing but generalization of instantaneous impulses in terms of effective time.
Various impulsive differential systems are studied in the books \cite{book hindawi,de gruyter} and papers \cite{rp ag1,rp ag main,donal jmaa,bala nonlinear,proc am,topo 2015,min zhang}. The fractional differential systems including Caputo derivative with non-instantaneous impulses are addressed by \cite{rp ag1}. %\textcolor{blue}{Extra amount of effort and carefulness is to be poured while dealing with non-instantaneous impulsive systems governed by Riemann-Liouville fractional derivatives in terms of initial conditions.}
% The books \cite{book hindawi,de gruyter} and articles \cite{rp ag1,rp ag main,donal jmaa,topo 2015,min zhang} contribute to the study of various impulsive differential systems. \cite{rp ag1} addresses the fractional differential systems involving Caputo derivative with non instantaneous impulses. 
%Very few articles (\cite{r p ag2}, \cite{min zhang}) discusses Riemann-Liouville problems with non-instantaneous impulses upto integral solution. %To the best of the author's knowledge,
There is currently no article that addresses the existence of mild solutions
 %no such article in the literature so far focusing the 
%addressing the analysis for approximate controllability 
for higher order Riemann-Liouville fractional evolution systems having non-instantaneous impulses, and hence, is the motivation for the present artefact.\\ 
%study
\indent The study of this article revolves around the following system:
% The motive of this paper is to investigate the approximate controllability of the following semilinar control systems involving Riemann-Liouville fractional derivatives:
\begin{eqnarray}
&&_{0}D_{t}^{\beta} z(t) = Az(t) + h(t,z(t)), ~~~~~ t\in \cup_{j=0}^{m} (u_j, t_{j+1}], \label{system}\\ \nonumber
&&z(t)= \varphi_j(t,z(t_{j}^{-})), ~~~ t\in (t_j, u_j],~~~ j=1,2,\ldots,m,\\ \nonumber
&&_{u_j}I_{t}^{2-\beta}z(t)|_{t=u_j} = \frac{1}{u_j-t_j}\int_{t_j}^{u_j} \varphi_j(r,z(t_{j}))dr,~~~ j=1,2,\ldots,m,\\ \nonumber
&&_{0}I_{t}^{2-\beta}z(t)|_{t=0} = z_{0} \in Z, \\ \nonumber 
&&_{u_j}D_{t}^{\beta-1}z(t)|_{t=u_j} = \widetilde{z}_j \in Z, ~~~ j=0,1,\ldots,m,\\ \nonumber
%&&_{0}D_{t}^{\beta-1}z(t)|_{t=0} = z_1 \in Z
\end{eqnarray}
where $_{0}D_{t}^{\beta}$ stands for the Riemann-Liouville fractional derivative of order $\beta \in (1,2)$ with fixed lower limit as 0. For each fixed $t$, $z(t)$ belong to Banach space $Z$. $ A:D(A)\subseteq Z\rightarrow Z$ generates Riemann-Liouville $ \beta$- order fractional resolvent $\mathcal{R}_\beta (t)$. 
%$B: L^q([0,a];V) \rightarrow L^q([0,a];Z)$ is a linear map. $h$ is a function from $[0,a] \times Z$ to $Z$.
The points $u_j$ and $t_j$ satisfy the relation $0 = u_0<t_1 < u_1 < t_2 < ... < u_m < t_{m+1} = u_{m+1} = a$. The impulses start at points $t_j, ~j = 1,2,\ldots,m$ and continue for the interval $(t_j,u_j]$. For $j=1,2,\ldots,m$, $ \varphi_j$ are the impulsive functions to be discussed later. $z(t_j)=z(t_j^-) = lim_{\triangle \rightarrow 0^+}z(t_j - \triangle) ~\text{and} ~z(t_j^+) = lim_{\triangle \rightarrow 0^+}z(t_j + \triangle)$ denotes the left and right hand limit of $z(t)$ at $t_j$ for $j=1,2,\ldots,m$ respectively. \\
%\indent \textcolor{blue}{Principal contributions of this artefact are:
%\begin{itemize}
%\item 
%Detailed existence analysis for the rarely studied non-instantaneous impulsive systems governed by Riemann-Liouville fractional derivatives is performed. 
%\item Correct integral type initial conditions are defined at starting and end points of the impulses. 
%\item The definition of mild solution is derived in terms of fractional resolvent $\mathcal{R}_{\beta}(t)$ paying attention towards correct lower limits of the Riemann-Liouville derivatives involved.
%\item Existence and uniqueness result for the mild solution of the concerned system is established under sufficient assumptions through Banach's fixed point theorem. 
%\item An example is provided at the end incorporating particular values of the input parameters to clearly illustrate the assumptions and methodology used. 
%\end{itemize}}

%%\indent This artefact is drafted as: Section \ref{2} introduces the preliminary concepts and definitions serving as the base for further study. Mild solution of the concerned system is constructed in terms of fractional resolvent in Section \ref{3}. Results for the existence and uniqueness of mild solutions based on appropriate sufficient assumptions are apparent in Section \ref{4}. Section \ref{5} comprises of an example to illustrate the methodology used. The article is winded up with concluding remarks and future scope in Section \ref{6}. 

\section{Preliminary facts} \label{2}
This segment introduces some definitions and fundamental concepts that will aid in the study of the artefact.

%This segment provides a quick referral to some fundamental concepts and definitions which are beneficial for the smooth study of the paper.\\
\indent Consider the Banach space 
\begin{eqnarray*}
 PC_{2-\beta}([0,a];Z) = &&\Big{ \{ }  z:z \in C \Big( \big( \cup_{j=0}^{m}(u_j,t_{j+1}) \big) \cup \big(\cup_{r=1}^{m}(t_j,u_j)\big); Z \Big),\\
 %PC_{1-\beta}([0,a];Z) = \{ z: (t-u_j)^{1-\beta}z(t) ~\text{is piecewise continuous on}~ [0,a],\\
&&~  z(u_j) = z(u_{j}^{-}) = lim_{\triangle \rightarrow 0^+} z(u_j - \triangle ) < \infty,~ j=1,2,\ldots,m,\\
&&~ (t-u_j)^{2-\beta}\|z(t)\|<\infty,~ \text{for}~ t\in (u_j,u_{j+1}], ~j=0,1,\ldots,m,\\
%&& ~ (t-p_m)^{1-\beta}\|z(t)\|<\infty,~ \text{for}~ t\in (p_m,t_{m+1}], \\
&&~ z(t_j) = z(t_j^-) = lim_{\triangle \rightarrow 0^+} z(t_j - \triangle ) < \infty, ~ j = 1,2,\ldots,m\Big{ \} }.
\end{eqnarray*} 
\indent Introduce the norm $\|z\|_{[0,a]}$ on $PC_{2-\beta}([0,a];Z)$ as 
%\begin{eqnarray*}
$\|z\|_{[0,a]}=\text{max}_{j=0,1,\ldots,m} \|z\|_j$
where
$\|z\|_j = \text{sup}_{t\in (u_j,u_{j+1}]} (t-u_j)^{2-\beta}\|z(t)\| ~\text{for}~ j=0,1,\ldots,m $.
% Clearly for $r=m$, the interval $(p_m,p_{m+1}]$ is nothing but $(p_m,t_{m+1}]$, i.e. $p_{m+1} = a = t_{m+1}$.
%and $\|z\|_m = \text{sup}_{t\in (p_m,a]}(t-p_m)^{1-\beta}\|z(t)\|$.

%Let for $t \geq 0$, $C_0-$ semigroup $T(t)$ has $ M $ as $ M = \text{sup}_{t \in [0,a]} \|T(t)\| < \infty $.\\
%\end{eqnarray*}
\begin{remark}
$PC_{2-\beta}([0,a];Z)$ is  a dense subset of $L^q([0,a];Z)$ if $q<\frac{1}{2-\beta}$.
\end{remark}
 Throughout this article, it is assumed that a constant $\lambda_\mathcal{R}$ exists such that $\|t^{2-\beta}\mathcal{R}_\beta(t)\| \leq \lambda_\mathcal{R} $ and $\tau = \underset{j= 0,1,\ldots,m}{max} (t_{j+1} - u_j)$.\\
\indent Some definitions related to fractional integrals and derivatives are :
\begin{definition} 
The Riemann-Liouville $\beta^{th}$-order fractional integral is written in terms of the following integral
\begin{equation*}
_{t_0}I^{\beta}_{t}z(t)=\frac{1}{\Gamma \left( \beta \right)} \int_{t_0}^{t} (t-r)^{\beta - 1}z(r)dr, \hspace{6mm} \beta > 0,
\end{equation*}
where $\Gamma$ denotes the gamma function. 
\end{definition}

\begin{definition}The fractional $\beta^{th}$-order Riemann-Liouville derivative is defined by the following expression
\begin{equation*}
_{t_0}D^{\beta}_{t}z(t)=\frac{1}{\Gamma \left( n - \beta \right)} {\Bigg{(} \frac{d}{dt} \Bigg{)}}^n \int_{t_0}^{t} (t-r)^{n- \beta - 1}z(r)dr,
\end{equation*}
where $0 \leq n-1 < \beta < n$.
\end{definition}
\begin{definition}
A function of the complex variable $w$ defined by
\begin{eqnarray}
E_{\beta, \beta} (w) = \sum_{i=0}^{\infty} \dfrac{w^i}{\Gamma(\beta i + \beta)} \nonumber
\end{eqnarray}
is known as the two parameter Mittag-Leffler series.
\end{definition}

\begin{lemma}{\cite{kilbas:srivast}} For $\beta > 0$, $k = \lfloor \beta \rfloor + 1 $ and $_{t_0}I_{t}^{k-\beta}z(t)=z_{k-\beta}(t)$, the following equality holds:
\begin{eqnarray*}
_{t_0}I_{t}^{\beta}{_{t_0}D_{t}^{\beta}z(t)}=z(t)-\sum_{i=1}^{k}\dfrac{z_{k-\beta}^{k-i}(t_0)}{\Gamma(\beta-i+1)}{(t-t_0)}^{\beta-i}.
\end{eqnarray*}
%\begin{itemize}
%\item[(i)] If $f \in C(t_0,a]$, then for any point $t\in (t_0,a]$ $$_{t_0}D_{t}^{\beta}(_{t_0}I_{t}^{\beta}f(t)) = f(t).$$
%\item[(ii)] If $f \in C(t_0,a]$ and $_{t_0}I_{t}^{1-\beta}f(t) \in C(t_0,a]$, then for any point $t \in (t_0,a]$ $$_{t_0}I_{t}^{\beta}(_{t_0}D_{t}^{\beta}f(t)) = f(t) - \dfrac{_{t_0}I_{t}^{1-\beta}f(t)|_{t=t_0}}{\Gamma(\beta)}(t-t_0)^{\beta-1}.$$
%\end{itemize}
\end{lemma}

\begin{proposition}{\cite{podlubny}}
The underneath hold true:
% for $ \alpha > 0$, $0<\beta<1$ :
\begin{itemize}
\item[(i)]  For $ \alpha > 0$, $0<\beta<1$, 
\begin{eqnarray*}
_{t_0}I_{t}^{\beta}(t-t_0)^{\alpha-1} &=& \dfrac{\Gamma(\alpha)}{\Gamma(\alpha + \beta)}(t-t_0)^{\alpha + \beta - 1}, \\
_{t_0}D_{t}^{\beta}(t-t_0)^{\alpha -1} &=& \dfrac{\Gamma(\alpha)}{\Gamma(\alpha-\beta)}(t-t_0)^{\alpha - \beta - 1}.
\end{eqnarray*}
\item[(ii)] For $0<\beta <1$,
\begin{eqnarray*}
_{t_0}I_{t}^{\beta}(t-t_0)^{-\beta} &=& \Gamma(1-\beta),\\
_{t_0}D_{t}^{\beta}(t-t_0)^{\beta-1}&=&0.
\end{eqnarray*}
\end{itemize}
\end{proposition}

\begin{definition}{\cite{fracresolvent}}
A family $\{ \mathcal{R}_\beta(t)| t> 0 \} \subset \mathcal{L}(Z)$ is called Riemann-Liouville $\beta$-order fractional resolvent if the underneath hold true:
\begin{itemize}
\item[(i)] for any $z \in Z$, $\mathcal{R}_\beta(\cdot)z \in C((0, \infty);Z)$ and 
$$\lim_{t \rightarrow 0^+} \Gamma (\beta - 1) t^{2-\beta} \mathcal{R}_\beta(t)z=z,$$
\item[(ii)] for $r,t > 0$, $\mathcal{R}_\beta(r)$ and $\mathcal{R}_\beta(t)$ commute, 
\item[(iii)] for $r,t >0$, 
$$\mathcal{R}_\beta(r) I_{t}^{\beta}\mathcal{R}_\beta(t)-I_{r}^{\beta}\mathcal{R}_\beta(r)\mathcal{R}_\beta(t) = \frac{r^{\beta-2}}{\Gamma(\beta-1)}I_{t}^{\beta} \mathcal{R}_\beta(t) - \frac{t^{\beta-2}}{\Gamma(\beta-1)}I_{r}^{\beta} \mathcal{R}_\beta(r).$$
\end{itemize}
\end{definition}

\begin{definition}
%{\cite{fracresolvent}}
The operator $A$, generator of Riemann-Liouville $\beta$-order fractional resolvent $\mathcal{R}_\beta(t)$ and its domain $D(A)$ are defined respectively as:
\begin{eqnarray*}
Az=\Gamma(2\beta-1) \lim_{t\rightarrow 0^+}\frac{t^{2-\beta} \mathcal{R}_\beta(t)z-\frac{1}{\Gamma(\beta-1)}z}{t^\beta} ~\mbox{for} ~z\in D(A),
\end{eqnarray*}
and
\begin{eqnarray*}
D(A)= \bigg\{ z \in Z: \lim_{t \rightarrow 0^+} \frac{t^{2-\beta} \mathcal{R}_\beta(t)z-\frac{1}{\Gamma(\beta-1)}z}{t^\beta} \mbox{exists} \bigg\}.
\end{eqnarray*}
\end{definition}

\begin{proposition}{\cite{fracresolvent}}
Let $\{ \mathcal{R}_\beta(t)| t> 0 \}$ be the Riemann-Liouville $\beta$-order fractional resolvent with $A$ being its generator, then :
\begin{itemize}
\item[(i)] For each $z\in D(A)$,
$$\mathcal{R}_\beta(t)z \in D(A) ~~\text{and}~~ A\mathcal{R}_\beta(t)z=\mathcal{R}_\beta(t)Az, $$
\item[(ii)]   for $t>0$, $z\in Z$, $$\mathcal{R}_\beta(t)z=\frac{t^{\beta-2}}{\Gamma(\beta-1)z}+AI_{t}^{\beta}\mathcal{R}_\beta(t)z,$$ 
\item[(iii)] for $t>0$, $z\in D(A)$, $$\mathcal{R}_\beta(t)z=\frac{t^{\beta-2}}{\Gamma(\beta-1)z}+I_{t}^{\beta}\mathcal{R}_\beta(t)Az,$$
\item[(iv)] the operator A is closed and densely defined. 
\end{itemize}
\end{proposition}

\section{Construction of mild solution} \label{3}
In order to construct the mild solution $z(t)$ for the system (\ref{system}), we proceed with the discussion below:\\
\textbf{Case 1:} For $t\in (0,t_1]$ (see {\cite{abdul higher order}}),
\begin{eqnarray*}
z(t) = \mathcal{R}_\beta(t)z_0 + \int_{0}^{t}\mathcal{R}_\beta(r)\widetilde{z}_0dr + \int_{0}^{t} \int_{0}^{t-r}\mathcal{R}_\beta(\phi) h(r,z(r))d\phi dr.
%\frac{t^{\beta-2}}{\Gamma(\beta-1)}z_0 + \frac{t^{\beta-1}}{\Gamma(\beta)}z_1 + A _{0}I_{t}^{\beta}
%t^{\beta-1}T_{\beta}(t)z_0 + \int_{0}^{t}(t-s)^{\beta-1}T_{\beta}(t-s)\big{[}Bu(s)+ h(s,z(s))\big{]}ds.
\end{eqnarray*}
\textbf{Case 2:} For $t \in (t_1,u_1]$,
\begin{eqnarray*}
z(t) &=& \varphi_1(t,z(t_1^-)) \nonumber\\
&=& \varphi_1\bigg{(}t, \mathcal{R}_\beta(t_1)z_0 + \int_{0}^{t_1}\mathcal{R}_\beta(r)\widetilde{z}_0dr + \int_{0}^{t_1} \int_{0}^{t_1-r}\mathcal{R}_\beta(\phi) h(r,z(r))d\phi dr\bigg{)}
\end{eqnarray*}
\textbf{Case 3:} For $t \in (u_1,t_2]$,
\begin{align*}
_{0}D_{t}^{\beta}z(t) ={}& \dfrac{1}{\Gamma(2-\beta)} \dfrac{d^2}{dt^2} \int_{0}^{t}(t-r)^{2-\beta - 1} z(r) dr\\
={}& \dfrac{1}{\Gamma(2-\beta)} \dfrac{d^2}{dt^2} \int_{0}^{t_1}(t-r)^{1-\beta} z(r) dr + \dfrac{1}{\Gamma(2-\beta)} \dfrac{d^2}{dt^2} \int_{t_1}^{u_1}(t-r)^{1-\beta} \varphi_1(r,z(t_1^-)) dr\\
&\quad+ \dfrac{1}{\Gamma(2-\beta)} \dfrac{d^2}{dt^2} \int_{u_1}^{t}(t-r)^{1-\beta} z (r) dr\\
={}& \dfrac{\beta (\beta-1)}{\Gamma(2-\beta)}\int_{0}^{t_1}\dfrac{z(r)}{(t-r)^{1+\beta}}dr + \dfrac{\beta (\beta -1)}{\Gamma(2-\beta)}\int_{t_1}^{u_1}\dfrac{\varphi_1(r,z(t_1^-))}{(t-r)^{1+\beta}}dr +~ _{u_1}D_{t}^{\beta}z(t)\\
={}&  \Phi_1(t,z(t)) +~ _{u_1}D_{t}^{\beta}z(t)
\end{align*}
Thus,
\begin{eqnarray}
&&_{u_1}D_{t}^{\beta}z(t) =  Az(t) + h(t,z(t)) -  \Phi_1(t,z(t)) \nonumber\\
&&_{u_1}I_{t}^{2-\beta}z(t)|_{t=u_1} = \frac{1}{u_1-t_1}\int_{t_1}^{u_1} \varphi_1(r,z(t_{1}))dr \nonumber \\
&&_{u_1}D_{t}^{\beta-1}z(t)|_{t=u_1} = \widetilde{z}_1 \nonumber
\end{eqnarray}
Hence, for $t \in (u_1,t_2]$
\begin{eqnarray*}
z(t) &=& \frac{\mathcal{R}_\beta(t-u_1)}{u_1-t_1}\int_{t_1}^{u_1}\varphi_1(r,z(t_1))dr + \int_{u_1}^{t}\mathcal{R}_\beta(r-u_1)\widetilde{z}_1dr\nonumber\\
&&~~+ \int_{u_1}^{t} \int_{0}^{t-r}\mathcal{R}_\beta(\phi)\big(h(r,z(r))-\Phi_1(r,z(r))\big)d\phi dr.
\end{eqnarray*}
\indent Continuing this process for each $j=2,3,\ldots,m$ and taking lower limit of Riemann-Liouville derivative as $u_j$ for each interval $(u_j,t_{j+1}]$, we define the mild solution as follows:

\begin{definition}
A function $z \in PC_{2-\beta}([0,a];Z)$ is called a mild solution of the system (\ref{system}) if it satisfies the following integral equation:
\begin{eqnarray}
z(t) = 
\begin{cases}
      \mathcal{R}_\beta(t)z_0 + \int_{0}^{t}\mathcal{R}_\beta(r)\widetilde{z}_0dr + \int_{0}^{t} \int_{0}^{t-r}\mathcal{R}_\beta(\phi)h(r,z(r))d\phi dr ~~~~~~\text{for} \ t \in (0,t_1],\\
      \varphi_j(t,z(t_{j}^{-})),  \hspace{2cm}\text{for} \ t \in (t_j,u_j], ~~~ j=1,2,\ldots,m \\
    \frac{\mathcal{R}_\beta(t-u_j)}{u_j-t_j}\int_{t_j}^{u_j}\varphi_j(r,z(t_j))dr + \int_{u_j}^{t}\mathcal{R}_\beta(r-u_j)\widetilde{z}_jdr + \int_{u_j}^{t} \int_{0}^{t-r}\mathcal{R}_\beta(\phi)\big( h(r,z(r)) \nonumber \\
 -\Phi_j(r,z(r))\big)d\phi dr  \hspace{2cm} \text{for}\ t\in (u_j,t_{j+1}], ~~j=1,\ldots,m 

     %& \text{for}\ t\in (p_i,t_{i+1}], ~~~i = 0,1,2...,n 

    \end{cases}
 \label{mildsol}
\end{eqnarray}
where
\begin{eqnarray}
%&\varphi_0 = z_0 , \nonumber \\
%\nonumber \\
%&\Phi_0 = 0,   \label{phidef}\\
%\nonumber \\
&\Phi_j(t,z(t)) = \dfrac{\beta (\beta -1)}{\Gamma(2-\beta)}\displaystyle{\sum_{k=0}^{j-1}}\bigg{(} \displaystyle{\int_{u_k}^{t_{k+1}}}\frac{z(r)}{(t-r)^{1+\beta}} dr + \displaystyle{\int_{t_{k+1}}^{u_{k+1}}} \dfrac{\varphi_{k+1}(r,z(t_{k+1}))}{(t-r)^{1+ \beta}} dr \bigg{)},  \label{phidef}
\end{eqnarray}
for $ j=1,2,...,m$.
% and
%\begin{eqnarray}
%T_{\beta}(t)&=&\beta \int_{0}^{\infty} \theta \xi_{\beta}(\theta)T(t^{\beta} \theta)d\theta ,\\
%\xi_{\beta}(\theta)&=&\dfrac{1}{\beta} \theta^{-1-\frac{1}{\beta}}\varpi_\beta{(\theta^{-\frac{1}{\beta}})}, \nonumber\\
%\varpi_\beta{(\theta)}&=& \dfrac{1}{\pi}{\sum_{n=1}^{\infty}}(-1)^{n-1}{\theta}^{-n\beta-1}\frac{\Gamma(n\beta + 1)}{n!}sin(n\pi \beta), \hspace{4mm} \theta \in (0,\infty), \nonumber
%\end{eqnarray}
%and $\xi_{\beta}$ is a probability density function defined on $(0,\infty)$, that is,\\
%
%$\xi_{\beta}(\theta) \geq 0$\; and\; $\int_{0}^{\infty} \xi_{\beta}(\theta)d\theta = 1$.
\end{definition}

%\begin{definition}
%Let $z(t,v)$ be a mild solution of the system (\ref{system}) at time t corresponding to a control $v(\cdot)\in L^q([0,a];V)$. The fractional impulsive system (\ref{system}) is said to be approximately controllable on $[0,a]$ if there exists a control $v(\cdot)\in L^q([0,a];V)$ which steers the system (\ref{system}) from an arbitrary initial state $z_0 \in Z$ to an $\epsilon-$neighbourhood of desired final state $z_f \in Z$ at final time $t=a$, i.e., $\|z(a,v)-z_f\|_Z < \epsilon$.
%%Let $z(t,v)$ be the mild solution of the impulsive system (\ref{system}) at time $t$ corresponding to a control $v(\cdot)$
%\end{definition}
%
%\begin{definition}
%The fractional impulsive system (\ref{system}) is said to be total approximately controllable on $[0,a]$, if for desired final state $z_{f_j} \in Z$,  the system (\ref{system}) is approximately controllable on every sub-interval $[u_j,t_{j+1}]$ of $[0,a]$ for $j=0,1,\ldots,m$, i.e., $\|z(u_j,v_j)-z_{f_j}\|_Z <\epsilon$.
%\end{definition}
%\begin{definition}
%Let $z(t,v)$ be a mild solution of the system (\ref{system}) at time t corresponding to a control $v(\cdot)\in L^q([0,a];V)$. The set $K_a(h)=\lbrace z(a,v) \in Z; v(\cdot)\in L^q([0,a];V \rbrace$ is called the reachable set for final time a. If $K_a(h)$ becomes dense in Z, the system (\ref{system}) is approximately controllable on [0,a].
%\end{definition}

\section{Existence of mild solution} \label{4}
This segment establishes the existence and uniqueness of mild solution for the system (\ref{system}) utilizing the Banach fixed point theorem. The results are  based on the below mentioned assumptions:
%In order to prove existence and uniqueness of the system (\ref{system}), the following assumptions are made:
\begin{itemize}
\item[(A1)] $q < \frac{1}{2-\beta}$.
\item[(A2)] A constant $\lambda_h>0$ exists 
%in a way 
satisfying \\
$$ \|h(t,z)-h(t,y)\|_Z \leq \lambda_h \|z-y\|_Z~~\forall ~z,y \in Z. $$
%\item[(H2)] A function $\varsigma(.) ~\mbox{in}~ L^q([0,a];\mathbb{R^+}), ~q>\frac{1}{\beta}$, and a constant $d>0$ exists such that\\
%%\begin{eqnarray*}
%$$\|h(t,z)\|_Z \leq \sigma(t) + \rho (t-u_j)^{2-\beta}\|z\|_{Z}$$ for a.e. $t\in (u_j,u_{j+1}]; j=0,1,\ldots,m ~ \mbox{and}~ \forall ~ z \in Z$. 
\item[(A3)] For $j=1,2,\ldots,m$, the impulsive functions $\varphi_{j}$ defined from $[t_j,u_j] \times Z$ to $Z$ are continuous and there exist constants $\lambda_{\varphi_j} \in (0,1]$, such that \\
%\begin{eqnarray*}
$$\| \varphi_j(t,z) - \varphi_j(t,y) \|_Z \leq \lambda_{\varphi_j} \|z-y\|_Z.$$
%\end{eqnarray*}
\item[(A4)] For $j=1,2,\ldots,m$, the impulsive functions $\varphi_{j}$ defined from $[t_j,u_j] \times Z$ to $Z$ are of the order $(u_j-t)^{1+\beta}$, i.e., $$\varphi_j(t,z) = (u_j-t)^{1+\beta}\psi_j(t,z),$$ where the functions $\psi_j$ are such $\|\psi_j(t,z) - \psi_j(t,y)\|_Z \leq \lambda_{\psi_j}\|z-y\|_Z$ for all $z,y \in Z$ and $\lambda_{\psi_j} \in (0,1]$.
%The impulsive functions $\varphi_{j}(\cdot, Z)$ are differentiable on $[t_j,u_j]; j=1,2,\ldots,m$ for every $z \in Z$ and constants $\lambda_{\varphi_j^{'}} \in (0,1]$ exists satisfying 
%$$\| \varphi_j^{'}(t,z) - \varphi_j^{'} (t,y) \|_Z \leq \lambda_{\varphi_j^{'}} \|z-y\|_Z.$$
\item[(A5)] The constant $c < 1$, ~
where
\begin{eqnarray*}
c &= \dfrac{\lambda_\mathcal{R} \lambda_{\varphi_j}}{ (t_j - u_{j-1})^{2-\beta}} + \dfrac{\lambda_\mathcal{R} \lambda_h \left(\Gamma(\beta-1)\right)^2 (t_{j+1}-u_j)^{\beta}}{\Gamma(2\beta-1)} + \dfrac{\lambda_\mathcal{R}}{(\beta-1)\Gamma(2-\beta)} \sum_{k=0}^{j-1}\left( \dfrac{t_{k+1} - u_k}{u_j - t_{k+1}} \right)^{\beta-1} \nonumber \\
&+ \dfrac{\lambda_\mathcal{R}}{\beta \Gamma(2-\beta)} \sum_{k=0}^{j-2} \dfrac{\lambda_{\varphi_{k+1}}{(t_{j+1} - u_j)^2}}{{(t_{k+1} - u_k)}^{2-\beta}{(u_j - u_{k+1})}^\beta} + \dfrac{\lambda_\mathcal{R}\lambda_{\psi_j}(t_{j+1}-u_j)^2}{\Gamma(2-\beta)(t_j-u_{j-1})^{2-\beta}}. \nonumber
\end{eqnarray*}
\end{itemize}
\begin{theorem}
The nonlinear system (\ref{system}) admits a unique mild solution in $PC_{2-\beta}([0,a];Z)$ 
%for each control $v(\cdot) \in L^q([0,a];V)$ 
under the assumptions $(A1)-(A5)$.
\end{theorem}

\begin{proof}
Consider the operator $\Xi$ as 
\begin{eqnarray}
(\Xi z)(t) = 
 \begin{cases}
      \mathcal{R}_\beta(t)z_0 + \int_{0}^{t}\mathcal{R}_\beta(r)\widetilde{z}_0dr + \int_{0}^{t} \int_{0}^{t-r}\mathcal{R}_\beta(\phi)h(r,z(r))d\phi dr ~~~~~~\text{for} \ t \in (0,t_1],\\
     
      \varphi_j(t,\Xi(z(t_{j}^{-})))  \hspace{2cm}\text{for} \ t \in (t_j,u_j], ~~~ j=1,2,\ldots,m, \\
    \frac{\mathcal{R}_\beta(t-u_j)}{u_j-t_j}\int_{t_j}^{u_j}\varphi_j(r,z(t_j))dr + \int_{u_j}^{t}\mathcal{R}_\beta(r-u_j)\widetilde{z}_jdr + \int_{u_j}^{t} \int_{0}^{t-r}\mathcal{R}_\beta(\phi)\big(h(r,z(r)) \\
 -\Phi_j(r,z(r))\big)d\phi dr  \hspace{2cm} \text{for}\ t\in (u_j,t_{j+1}], ~~j=1,\ldots,m .
 \end{cases}
\end{eqnarray}
%%The existence of unique mild solution of the system is proved by using generalisation of Banach fixed point theorem as follows:
\indent First, it is evident that $\Xi$ maps $PC_{2-\beta}([0,a];Z)$ into itself.
It is now required to prove $\Xi$ as a contraction operator on $PC_{2-\beta}([0,a];Z)$.\\
\indent Let $z,y \in PC_{2-\beta}([0,a];Z)$, then\\
\\

\textbf{Case 1:}~~ For $t \in (0,t_1]$
\begin{eqnarray}
t^{2-\beta}\big{\|} (\Xi z)(t) - (\Xi y)(t) \big{\|} &\leq & t^{2-\beta} \int_{0}^{t} \bigg{\|} \int_{0}^{t-r}\mathcal{R}_\beta(\phi)\big( h(r,z(r))-h(r,y(r)) \big) d\phi \bigg{\|}dr \nonumber \\
&\leq & \frac{\lambda_\mathcal{R} \lambda_h}{\beta-1}t^{2-\beta}\int_{0}^{t} (t-r)^{\beta-1}\|z(r)-y(r)\|_Z dr \nonumber \\
&\leq & \frac{\lambda_\mathcal{R} \lambda_h}{\beta-1}t^{2-\beta}\int_{0}^{t}(t-r)^{\beta-1}r^{\beta-2}r^{2-\beta}\|z(r)-y(r)\|_Z dr \nonumber \\
%&\leq & \frac{\lambda_\mathcal{R} \lambda_h}{\beta-1}t^{2-\beta}\|z-y\|_0 \int_{0}^{t}(t-r)^{\beta-1}r^{\beta-2}dr \nonumber \\
&\leq & \frac{\lambda_\mathcal{R} \lambda_h}{\beta-1}\frac{\Gamma(\beta)\Gamma(\beta-1)}{\Gamma(2\beta-1)}t^{2\beta-2}t^{2-\beta}\|z-y\|_0 \nonumber \\
&\leq & \frac{\lambda_\mathcal{R} \lambda_h (\Gamma(\beta-1))^2}{\Gamma(2\beta-1)}t^\beta \|z-y\|_0 \nonumber \\
& \leq & c \|z-y\|_{[0,a]} \label{C1}
\end{eqnarray}

\textbf{Case 2:}~~ For $t \in (u_j,t_{j+1}]$, ~~$j= 1,2,\ldots,m$.
\begin{eqnarray}
&&(t-u_j)^{2-\beta} \big{\|} (\Xi z)(t) - (\Xi y)(t) \big{\|} \nonumber \\
&& \quad \leq \bigg{\|} (t-u_j)^{2-\beta} \frac{\mathcal{R}_\beta(t-u_j)}{u_j-t_j}\int_{t_j}^{u_j}\varphi_j(r,z(t_{j})) - \varphi_j(r,y(t_{j})) dr \bigg{\|}  \nonumber \\
%&& \quad \quad + (t-u_j)^{2-\beta} \int_{u_j}^{t} \bigg{\|} \mathcal{R}_\beta(r-u_j) \big(\varphi^{'}(u_j,z(t_j))-\varphi^{'}(u_j,y(t_j))\big)\bigg{\|}dr \nonumber \\
&&\quad \quad + (t-u_j)^{2-\beta} \int_{u_j}^{t} \bigg{\|}\int_0^{t-r}\mathcal{R}_\beta (\phi)\big{(} h(r,z(r)) - h(r,y(r)) \big{)}d\phi \bigg{\|} dr \nonumber\\
&&\quad \quad + (t-u_j)^{2-\beta} \int_{u_j}^{t}\bigg{\|}\int_0^{t-r}\mathcal{R}_\beta (\phi) \big{(} \Phi_j(r,z(r)) - \Phi_j(r,y(r)) \big{)} d\phi \bigg{\|} dr \nonumber  \\
&&\quad \leq  \frac{\lambda_\mathcal{R}\lambda_{\varphi_j}}{u_j-t_j} \int_{t_j}^{u_j}\big{\|} z(t_{j}) - y(t_{j}) \big{\|}_Z dr
%+ (t-u_j)^{2-\beta}\lambda_\mathcal{R} \lambda_{\varphi_j^{'}}\int_{u_j}^{t}(r-u_j)^{\beta-2}\big{\|} z(t_{j}) - y(t_{j}) \big{\|}_Z dr \nonumber \\
+ (t-u_j)^{2-\beta} \frac{\lambda_{\mathcal{R}}\lambda_h}{\beta-1} \int_{u_j}^{t} (t-r)^{\beta - 1}  \|z(r) - y(r)\|_Z dr \nonumber \\
&&\quad \quad + (t-u_j)^{2-\beta}\frac{\lambda_\mathcal{R}}{(\beta-1)} \int_{u_j}^{t} (t-r)^{\beta - 1} \big{\|} \Phi_j(r,z(r)) - \Phi_j(r,y(r)) \big{\|}_Z dr.  \label{halfeq}
\end{eqnarray}
\indent Solving $ (t-u_j)^{2-\beta} \int_{u_j}^{t} (t-r)^{\beta - 1} \big{\|} \Phi_j(r,z(r)) - \Phi_j(r,y(r)) \big{\|} dr$ seperately for $j=1,2,\ldots,m$ as for  we proceed with
\begin{eqnarray*}
\Phi_j(r,z(r)) = \dfrac{\beta (\beta -1)}{\Gamma(2-\beta)}\displaystyle{\sum_{k=0}^{j-1}}\bigg{(} \displaystyle{\int_{u_k}^{t_{k+1}}}\frac{z(s_1)}{(r-s_1)^{1+\beta}}ds_1 + \displaystyle{\int_{t_{k+1}}^{u_{k+1}}} \frac{\varphi_{k+1}(s_1,z(t_{k+1}^-))}{(r-s_1)^{1+ \beta}}ds_1 \bigg{)},~ j=1,2,\ldots,m.
\end{eqnarray*}
and for $j=0$, $\Phi_0 = 0$.\\
\indent So, basing our following evaluations for $r\in (u_j,t_{j+1}], ~ j=1,\ldots,m$ and using $r-u_k > r-t_{k+1}$
\begin{eqnarray}
\int_{u_k}^{t_{k+1}} \dfrac{\|z(s_1)-y(s_1)\|}{{(r-s_1)}^{\beta+1}} ds_1 &=& \int_{u_k}^{t_{k+1}} \dfrac{{(s_1 - u_k)}^{2-\beta} \|z(s_1)-y(s_1)\|}{{(s_1 - u_k)}^{2-\beta}{(r-s_1)}^{\beta+1}} ds_1 \nonumber \\
%& = & \dfrac{(t_{k+1}-u_k)^{\beta-1}(\beta r - t_{k+1}-(\beta-1)u_k)}{\beta (\beta -1) (r-u_k)^2 (r-t_{k+1})^\beta}\|z-y\|_k \nonumber \\
& = & \dfrac{(t_{k+1}-u_k)^{\beta-1}}{\beta (\beta -1) (r-u_k)^2 (r-t_{k+1})^\beta} [\beta (r-u_k)-(t_{k+1}-u_k)] \|z-y\|_k \nonumber \\
%& \leq & \dfrac{(t_{k+1}-u_k)^{\beta-1}\beta (r-u_k)}{ \beta(\beta -1) (r-u_k)^2 (r-t_{k+1})^\beta} \|z-y\|_k \nonumber \\
& \leq & \dfrac{(t_{k+1}-u_k)^{\beta-1}}{ (\beta -1) (r-u_k) (r-t_{k+1})^\beta} \|z-y\|_k \nonumber \\
& \leq & \dfrac{{(t_{k+1} - u_k)}^{\beta-1}}{(\beta-1) {(r-t_{k+1})}^{1+\beta } }\| z-y \|_k  \label{singint1}
\end{eqnarray}
\indent Using (\ref{singint1}) and $ t-t_{k+1} > t - u_j, ~ k=0,1,\ldots,j-1$, we obtain
\begin{eqnarray}
&&{(t-u_j)}^{2-\beta} \int_{u_j}^{t} \int_{u_k}^{t_{k+1}} {(t-r)}^{\beta-1} \dfrac{\|z(s_1)-y(s_1)\|}{{(r-s_1)}^{1+\beta}} ds_1 dr \nonumber \\
&&\quad \leq {(t-u_j)}^{2-\beta} \dfrac{{(t_{k+1} - u_k)}^{\beta-1} \| z-y \|_k}{(\beta-1 )}  \int_{u_j}^{t}  \dfrac{1}{{(t-r)}^{1-\beta}{(r-t_{k+1})}^{1+\beta }} dr \nonumber \\
&&\quad = {(t-u_j)}^{2-\beta}\dfrac{{(t_{k+1} - u_k)}^{\beta-1} \| z-y \|_k}{(\beta-1)} \dfrac{{(t-u_j)}^\beta}{\beta {(u_j - t_{k+1})}^{\beta} (t - t_{k+1})} \nonumber \\
&&\quad \leq \dfrac{(t_{j+1}-u_j){(t_{k+1} - u_k)}^{\beta-1}}{\beta (\beta-1){(u_j - t_{k+1})}^{\beta}} \|z-y\|_k. \nonumber \\
&& \quad \leq \left( \dfrac{t_{k+1} - u_k}{u_j - t_{k+1}} \right)^{\beta-1} \dfrac{\|z-y\|_k}{\beta (\beta - 1)} \label{doubint1}
%&&\quad \leq \dfrac{\|z-y\|_k}{\beta (\beta-1)} {\left( \dfrac{t_{k+1} - u_k}{u_j - t_{k+1}} \right)}^{\beta-1} . \label{doubint1}
\end{eqnarray}
\indent Next, for $k=0,1,\ldots,j-1$ and $j=1,2,\ldots,m$, we have
\begin{eqnarray}
&&\int_{t_{k+1}}^{u_{k+1}} \dfrac{\| \varphi_{k+1}(s_1,z(t_{k+1})) - \varphi_{k+1}(s_1,y(t_{k+1}))\|}{{(r-s_1)}^{\beta + 1}} ds_1 \nonumber \\
&&\quad \leq \lambda_{\varphi_{k+1}} \int_{t_{k+1}}^{u_{k+1}} \dfrac{{(t_{k+1}-u_k)}^{2-\beta} \| z(t_{k+1})-y(t_{k+1})\|}{{(t_{k+1}-u_k)}^{2-\beta}{(r-s_1)}^{\beta + 1}} ds_1 \nonumber \\
%&\leq & \dfrac{b_{j+1} \| z-y\|_{j+1} }{{(t_{j+1}-p_j)}^{1-\beta}} \Bigg{[} \dfrac{{(s-s_1)}^{-\beta-1+1}}{-\beta \times -1} {\Bigg{]}}_{t_{j+1}}^{p_{j+1}}\\
&&\quad \leq  \dfrac{\lambda_{\varphi_{k+1}} \| z-y\|_{k} }{\beta{(t_{k+1}-u_k)}^{2-\beta}} \Bigg{[} \dfrac{1}{{(r-u_{k+1})}^\beta} - \dfrac{1}{{(r-t_{k+1})}^{\beta}} {\Bigg{]}} \nonumber \\
&&\quad \leq \dfrac{\lambda_{\varphi_{k+1}} \| z-y\|_{k} }{\beta{(t_{k+1}-u_k)}^{2-\beta}} . \dfrac{1}{{(r-u_{k+1})}^\beta} \label{singint2}
%&&\quad \leq \dfrac{\lambda_{\varphi_{k+1}} (t_{k+1} - u_k)^{\beta-1}}{\beta (r-u_{k+1})^{\beta+1}} \|z-y\|_k . \label{singint2}
\end{eqnarray}
\indent Using (\ref{singint2}) for $k=0,1,2,\ldots,j-2$ and $j=2,3,\ldots,m$, we obtain
\begin{eqnarray}
&&{(t-u_j)}^{2-\beta} \int_{u_j}^{t} \int_{t_{k+1}}^{u_{k+1}} {(t-r)}^{\beta-1} \dfrac{\| \varphi_{k+1}(s_1,z(t_{k+1})) - \varphi_{k+1}(s_1,y(t_{k+1}))\|}{{(r-s_1)}^{\beta + 1}} ds_1 dr \nonumber \\
&&\quad \leq  \dfrac{\lambda_{\varphi_{k+1}} \| z-y\|_{k} }{\beta {(t_{k+1} - u_k)}^{2-\beta}} {(t-u_j)}^{2-\beta} \int_{u_j}^{t} \dfrac{1}{{{(r-u_{k+1})}^{\beta}}{(t-r)}^{1-\beta}} dr \nonumber \\
&&\quad \leq \dfrac{\lambda_{\varphi_{k+1}} \| z-y\|_{k} }{\beta {(t_{k+1} - u_k)}^{2-\beta}} {(t-u_j)}^{2-\beta} \dfrac{{(t-u_j)}^{\beta}}{\beta {(u_j - u_{k+1})}^\beta} \nonumber \\
&&\quad \leq \dfrac{\lambda_{\varphi_{k+1}}{(t_{j+1} - u_j)^2}}{\beta^2{(t_{k+1} - u_k)}^{2-\beta}{(u_j - u_{k+1})}^\beta} \| z-y\|_{k} \label{doubint2}
\end{eqnarray}
and for $k=j-1$ and $j=1,2,\ldots,m$, it is 
\begin{eqnarray}
&&{(t-u_j)}^{2-\beta} \int_{u_j}^{t} \int_{t_{j}}^{u_{j}} {(t-r)}^{\beta-1} \dfrac{\| \varphi_{j}(s_1,z(t_{j})) - \varphi_{j}(s_1,y(t_{j}))\|}{{(r-s_1)}^{\beta + 1}} ds_1 dr \nonumber\\
&&\quad \leq {(t-u_j)}^{2-\beta} \int_{u_j}^{t} \int_{t_{j}}^{u_{j}} {(t-r)}^{\beta-1} \dfrac{(u_j - s_1)^{1+\beta} \| \psi_j(s_1,z(t_j))-\psi_j(s_1,y(t_j))\|}{ (r-s_1)^{1+\beta}}ds_1 dr \nonumber \\
&&\quad \leq \lambda_{\psi_j} {(t-u_j)}^{2-\beta}  \int_{u_j}^{t} \int_{t_{j}}^{u_{j}} {(t-r)}^{\beta-1} \dfrac{(u_j - s_1)^{1+\beta} (t_j-u_{j-1})^{2-\beta}\| z(t_j)-y(t_j)\|}{(t_j-u_{j-1})^{2-\beta} (r-s_1)^{1+\beta}}ds_1 dr \nonumber \\
&&\quad \leq \dfrac{\lambda_{\psi_j}  \|z-y\|_{j-1}}{(t_j-u_{j-1})^{2-\beta}}(t-u_j)^{2-\beta}  \int_{u_j}^{t} \int_{t_{j}}^{u_{j}}  \dfrac{{(t-r)}^{\beta-1} (u_j - s_1)^{1+\beta}}{ (r-s_1)^{1+\beta}}ds_1 dr \nonumber\\
&& \quad \leq \dfrac{\lambda_{\psi_j}  \|z-y\|_{j-1}}{(t_j-u_{j-1})^{2-\beta}}(t-u_j)^{2-\beta}  \int_{u_j}^{t} {(t-r)}^{\beta-1} dr \nonumber \\
&& \quad \leq \dfrac{\lambda_{\psi_j}  \|z-y\|_{j-1}}{(t_j-u_{j-1})^{2-\beta}}(t-u_j)^{2-\beta} \dfrac{(t-u_j)^\beta}{\beta} \nonumber \\
&& \quad \leq \dfrac{\lambda_{\psi_j}(t_{j+1}-u_j)^2 }{\beta (t_j-u_{j-1})^{2-\beta}} \|z-y\|_{j-1} \label{doubint3}
%&&\quad \leq {(t-u_j)}^{2-\beta} \dfrac{\lambda_{\varphi_{j}} \| z-y\|_{j-1} }{\beta{(t_{j}-u_{j-1})}^{2-\beta}} \int_{u_j}^{t} \dfrac{1}{{(t-r)}^{1-\beta}{{(r-u_{j})}^\beta}} dr \nonumber\\
%&&\quad = {(t-u_j)}^{2-\beta} \dfrac{\lambda_{\varphi_{j}} \| z-y\|_{j-1} }{\beta{(t_{j}-u_{j-1})}^{2-\beta}} B(\beta,1-\beta) \nonumber \\
%&&\quad \leq {(t_{j+1}-u_j)}^{2-\beta} \dfrac{\lambda_{\varphi_{j}} \| z-y\|_{j-1} }{\beta{(t_{j}-u_{j-1})}^{2-\beta}} \Gamma(\beta)\Gamma(1-\beta). \label{doubint3}
\end{eqnarray}
\begin{remark}
The integral $\int_{u_j}^{t} \int_{t_j}^{u_j} \dfrac{(t-r)^{\beta-1}(u_j-s_1)^{1+\beta}}{(r-s_1)^{1+\beta}}ds_1dr$ reduces to the integral $\int_{u_j}^{t}(t-r)^{\beta-1}dr$ by splitting the integral in the neighbourhood of $u_j$ and using the fact that $\dfrac{u_j-s_1}{r-s_1}<1$ and hence the assumption $(A4)$ is required. 
\end{remark}
\indent Combining equations (\ref{phidef}), (\ref{doubint1}), (\ref{doubint2}), and (\ref{doubint3}), we obtain
\begin{eqnarray}
&&(t-u_j)^{2-\beta} \int_{u_j}^{t} (t-r)^{\beta - 1} \big{\|} \Phi_j(r,z(r)) - \Phi_j(r,y(r)) \big{\|} dr \nonumber \\
&&\quad \leq \dfrac{\|z-y\|_{[0,a]}}{\Gamma(2-\beta)} \Bigg( \sum_{k=0}^{j-1}\left( \dfrac{t_{k+1} - u_k}{u_j - t_{k+1}} \right)^{\beta-1} + \left(\frac{\beta-1}{\beta}\right)\sum_{k=0}^{j-2} \dfrac{\lambda_{\varphi_{k+1}}{(t_{j+1} - u_j)^2}}{{(t_{k+1} - u_k)}^{2-\beta}{(u_j - u_{k+1})}^\beta} \nonumber \\
&& \quad + \dfrac{(\beta-1)\lambda_{\psi_j}(t_{j+1}-u_j)^2 }{(t_j-u_{j-1})^{2-\beta}}\Bigg). \label{normphi}
%\dfrac{\|z-y\|_{[0,a]}}{ \Gamma(2-\beta)} \Bigg{(} \dfrac{1}{\beta-1}  \sum_{k=0}^{j-1}  {\left( \dfrac{t_{k+1} - u_k}{u_j - t_{k+1}} \right)}^{\beta-1} + \dfrac{1}{\beta} \sum_{k=0}^{j-2} \dfrac{\lambda_{\varphi_{k+1}}{(t_{j+1} - u_j)^2}}{{(t_{k+1} - u_k)}^{2-\beta}{(u_j - u_{k+1})}^\beta} \nonumber \\
%&&\quad \quad + \dfrac{\lambda_{\varphi_{j}} (t_{j+1}-u_j)^{2-\beta}}{(t_{j}-u_{j-1})^{2-\beta}} \Gamma(\beta) \Gamma(1-\beta) \Bigg{)}. \label{normphi}
\end{eqnarray}

%Now, proceeding with eqn. (\ref{halfeq}) after using eqn. (\ref{normphi}) as follows 
Applying equation (\ref{normphi}) in equation (\ref{halfeq}), we proceed as 
\begin{eqnarray}
&&(t-u_j)^{2-\beta} \big{\|} (\Xi z)(t) - (\Xi y)(t) \big{\|} \nonumber \\
&&\quad \leq \dfrac{\lambda_\mathcal{R} \lambda_{\varphi_j}}{ (t_j - u_{j-1})^{2-\beta}}\| z - y \|_{j-1} + \dfrac{\lambda_\mathcal{R} \lambda_h(t-u_j)^{2-\beta}}{(\beta -1)}\|z-y\|_j \int_{u_j}^{t} (t-r)^{\beta - 1} (r-u_j)^{\beta - 2} dr \nonumber \\
&&\quad \quad + \dfrac{\lambda_\mathcal{R} \|z-y\|_{[0,a]}}{(\beta-1)\Gamma(2-\beta)} \Bigg( \sum_{k=0}^{j-1}\left( \dfrac{t_{k+1} - u_k}{u_j - t_{k+1}} \right)^{\beta-1} + \left(\frac{\beta-1}{\beta}\right)\sum_{k=0}^{j-2} \dfrac{\lambda_{\varphi_{k+1}}{(t_{j+1} - u_j)^2}}{{(t_{k+1} - u_k)}^{2-\beta}{(u_j - u_{k+1})}^\beta} \nonumber \\
&& \quad \quad + \dfrac{(\beta-1)\lambda_{\psi_j}(t_{j+1}-u_j)^2 }{(t_j-u_{j-1})^{2-\beta}}\Bigg) \nonumber \\
&&\quad \leq \Bigg[ \dfrac{\lambda_\mathcal{R} \lambda_{\varphi_j}}{ (t_j - u_{j-1})^{2-\beta}} + \dfrac{\lambda_\mathcal{R} \lambda_h \left(\Gamma(\beta-1)\right)^2 (t_{j+1}-u_j)^{\beta}}{\Gamma(2\beta-1)} + \dfrac{\lambda_\mathcal{R}}{(\beta-1)\Gamma(2-\beta)} \sum_{k=0}^{j-1}\left( \dfrac{t_{k+1} - u_k}{u_j - t_{k+1}} \right)^{\beta-1} \nonumber \\
&& \quad \quad + \dfrac{\lambda_\mathcal{R}}{\beta \Gamma(2-\beta)} \sum_{k=0}^{j-2} \dfrac{\lambda_{\varphi_{k+1}}{(t_{j+1} - u_j)^2}}{{(t_{k+1} - u_k)}^{2-\beta}{(u_j - u_{k+1})}^\beta} + \dfrac{\lambda_\mathcal{R}\lambda_{\psi_j}(t_{j+1}-u_j)^2}{\Gamma(2-\beta)(t_j-u_{j-1})^{2-\beta}}\Bigg] \|z-y\|_{[0,a]} \nonumber \\
&& \quad = c \|z-y\|_{[0,a]} \label{c1}
\end{eqnarray}
%Therefore, 
%\begin{eqnarray*}
%\|Gz-Gy\|_r = \underset{t \in (u_j,t_{j+1}]}{sup} (t-u_j)^{1-\beta} \|(Gz)(t)-(Gy)(t)\| 
%\leq \nu \|z-y\|_{[0,a]}
%\end{eqnarray*}

\textbf{Case 3:} ~~For $t \in (t_j,u_j]$,~~$j= 1,2,...,m$
\begin{eqnarray}
&&(t-u_{j-1})^{2-\beta} \big{\|} (\Xi z)(t) - (\Xi y)(t) \big{\|} \nonumber \\
&&\quad = (t-u_{j-1})^{2-\beta} \| \varphi_j(t,\Xi(z(t_{j})) - \varphi_j(t,\Xi(y(t_{j}))\| \nonumber \\
&& \quad\leq \lambda_{\varphi_j} (t-u_{j-1})^{2-\beta} \| \Xi(z(t_j)) - \Xi(y(t_j)) \|_Z \nonumber \\
&& \quad \leq \lambda_{\varphi_j} \Bigg[ \dfrac{\lambda_\mathcal{R} \lambda_{\varphi_j}}{ (t_j - u_{j-1})^{2-\beta}} + \dfrac{\lambda_\mathcal{R} \lambda_h \left(\Gamma(\beta-1)\right)^2 (t_{j+1}-u_j)^{\beta}}{\Gamma(2\beta-1)} + \dfrac{\lambda_\mathcal{R}}{(\beta-1)\Gamma(2-\beta)} \sum_{k=0}^{j-1}\left( \dfrac{t_{k+1} - u_k}{u_j - t_{k+1}} \right)^{\beta-1} \nonumber \\
&& \quad \quad + \dfrac{\lambda_\mathcal{R}}{\beta \Gamma(2-\beta)} \sum_{k=0}^{j-2} \dfrac{\lambda_{\varphi_{k+1}}{(t_{j+1} - u_j)^2}}{{(t_{k+1} - u_k)}^{2-\beta}{(u_j - u_{k+1})}^\beta} + \dfrac{\lambda_\mathcal{R}\lambda_{\psi_j}(t_{j+1}-u_j)^2}{\Gamma(2-\beta)(t_j-u_{j-1})^{2-\beta}}\Bigg] \| z - y \|_{[0,a]} \nonumber \\
&& \quad \leq c \|z-y\|_{[0,a]}. \label{c2}
\end{eqnarray}

Therefore, on combining equations (\ref{C1}), (\ref{c1}) and (\ref{c2}), we obtain 
\begin{eqnarray*}
\|\Xi z- \Xi y\|_{[0,a]} \leq c \| z - y \|_{[0,a]} 
%~~~ \mbox{for}~ r = 0,1,...,m.
\end{eqnarray*}
\indent Thus, $\Xi$ is contraction and it is evident through Banach fixed point theorem that $\Xi$ possess a unique fixed point $z(\cdot)$ in $PC_{2-\beta}([0,a];Z)$ which serves as the requisite solution of system (\ref{system}).
\end{proof}

\section{Example} \label{5}
\indent Consider the below mentioned initial value problem 
\begin{eqnarray}
&&_{0}D_{t}^{\beta}z(t,v) = \frac{\partial^2}{\partial v^2} z(t,v) + h(t,z(t,v)), ~t \in \cup_{j=0}^{3}(u_j,t_{j+1}] \subset [0,1],~v \in [0,\pi], \label{sysex1}  \\
&&z(t,v) = \varphi_j(t,z(t_j,v)) , ~~ t \in \cup_{j=1}^{3}(t_j,u_j], ~v\in [0,\pi], \nonumber \\
&&z(t,0) = 0 = z(t,\pi), ~ t \in (0,1], \nonumber \\
&&_{0}I_{t}^{2-\beta}z(t,v)|_{t=0} = z_0(v) ,~~v \in [0,\pi] \nonumber\\
&&_{u_j}I_t^{2-\beta}z(t,v)|_{t=u_j} = \dfrac{1}{u_j-t_j}\int_{t_j}^{u_j} \varphi_j(r, z(t_j,v))dr, ~~ v \in [0,\pi] \nonumber \\
&&_{u_j}D_t^{\beta-1}z(t,v)|_{t=u_j} = {z}_j^\prime(v) , ~~v \in [0,\pi], j=0,1,\ldots,3 \nonumber 
\end{eqnarray}
  %$\textcolor{blue}{\text{with~} u_0=0, t_1 = 0.22, u_1 = 0.26, t_2 = 0.48, u_2 = 0.52, t_3 = 0.74, u_3 = 0.78, t_4 = 1}$. \\
\indent Let $W = W^{'} = L^2[0,\pi]$ and the operator $A:D(A) \subset W \rightarrow W$ defined as 
\begin{eqnarray*}
Aw = w''
\end{eqnarray*}
where
\begin{eqnarray*}
D(A) =  &&\Bigg \{ w\in W ~\mid~ w, ~\frac{\partial w}{\partial v} \text{ ~are absolutely continuous},~ \frac{\partial^2 w}{\partial v^2} \in W\\
&&\quad \text{ and } w(0) = 0 = w(\pi) \Bigg \}.
\end{eqnarray*}
\indent Then the $\beta$- order fractional resolvent $\mathcal{R}_\beta(t)$ generated by 
\begin{eqnarray*}
\left(\mathcal{R}_\beta(t)w\right)(v)=\sum_{\gamma=1}^{\infty} t^{\beta-2} E_{\beta,\beta-1}(-\gamma^2 t^\beta)w_\gamma \sin(\gamma v),
\end{eqnarray*}
where $-\gamma^2$ are the eigenvalues of A and $\sin(\gamma v)$ be the corresponding eigenvectors respectively and $w(v)=\sum_{\gamma=1}^{\infty} w_\gamma \sin(\gamma v)$ (refer \cite{fracresolvent}).\\
The abstract form of the system (\ref{system}) is expressed as:
\begin{eqnarray}
&&_{0}D_{t}^{\beta}\widehat{z}(t) = A\widehat{z}(t) +  h(t,\widehat{z}(t)),~~ t \in (0,1],\\
&&\widehat{z}(t) = \varphi_j(t,\widehat{z}(t_j)) , ~~ t \in \cup_{j=1}^{3}(t_j,u_j], ~v\in [0,\pi], \nonumber \\
%&&z(t,0) = 0 = z(t,\pi), ~ t \in (0,1], \nonumber \\
&&_{0}I_{t}^{2-\beta}\widehat{z}(t)|_{t=0} = \widehat{z}_0  , \nonumber\\
&&_{u_j}I_t^{2-\beta}\widehat{z}(t)|_{t=u_j} = \dfrac{1}{u_j-t_j}\int_{t_j}^{u_j} \varphi_j(r, \widehat{z}(t_j))dr,  \nonumber \\
&&_{u_j}D_t^{\beta-1}\widehat{z}(t)|_{t=u_j} = \widehat{{z}_j}^\prime, j=0,1,\ldots,3. \nonumber 
\end{eqnarray}
where $\widehat{z}(t)=z(t,\cdot)$, $\widehat{z}_0=z_0(\cdot)$ and $\widehat{{z}_j}^\prime = {z}_j^\prime (v)$.\\
%%\indent Let us choose the nonlinear function $h$ as \\
%%$h(t,z(t,v))= 1+(t-u_j)^2 + \delta (t-u_j)^\eta [ z(t,v) + \sin z(t,v)]$,\\ where $t\in (u_j,u_{j+1}]$ and $j=0,1,\ldots,3$. $\delta$ and $\eta$ are constants, and $\eta \geq 2-\beta$.\\
%%\indent Now, 
%%\begin{eqnarray*}
%%\|h(t,z(t,v)) - h(t,y(t,v))\| &\leq & |\delta| (t-u_j)^\eta \| z(t,v)-y(t,v) + \sin z(t,v) - \sin y(t,v)\|\\
%%&\leq & |\delta| (t-u_j)^{\eta + \beta -2}(t-p_r)^{2 - \beta}\bigg[\|z(t,v)-y(t,v)\|\\
%% & \quad + & \Big{\|}2\cos\left(\frac{z(t,v)+y(t,v)}{2}\right)\sin \left(\frac{z(t,v)-y(t,v)}{2}\right)\Big{\|} \bigg]\\
%% & \leq & 2|\delta |(t-u_j)^{2 - \eta} \| z(t,v)-y(t,v) \|\\
%% & \leq & 2|\delta| \| z(t,v)-y(t,v) \|.
%%\end{eqnarray*}
%%\indent The assumption $(A2)$ is satisfied with $\lambda_h = 2|\delta |$.\\
%%\textcolor{blue}{Let us choose the impulsive functions $\varphi_j(t, z(t_j,v))$ as 
%%$\varphi_j(t, z(t_j,v)) = (u_j-t)^{1+\beta}t^2\|z(t_j,v)\|$,\\ for $t \in (t_j,u_j] , ~j = 1,\ldots,3.$\\
%%Now, 
%%\begin{eqnarray*}
%%\| \varphi_j(t, z(t_j,v)) - \varphi_j(t, y(t_j,v)) \| & \leq & (u_j-t)^{1+\beta}t^2\|z(t_j,v) - y(t_j,v)\| \\
%%& \leq & (u_j-t)^{1+\beta} \|z(t_j,v) - y(t_j,v)\| \\
%%& \leq & \ell \|z(t_j,v) - y(t_j,v)\|
%% \end{eqnarray*}
%%Clearly, the assumption $(A3)$ is satisfied with $\lambda_{\varphi_j} = \ell = (0.0016)^{1+\beta}$ and assumption $(A4)$ is satisfied with $\psi_j(t,z(t_j,v))=t^2 \|z(t_j,v))\|$.}
Assumption $(A5)$ is satisfied by choosing $\delta$ to be sufficiently close to zero. Thus, existence and uniqueness of the system (\ref{1}) is achieved from Theorem $4.1$.
% for $\beta > 3/2$.
%Thus, approximate controllability of (\ref{sysex2}) follows from Theorem $2$ without assuming $T(t)$ as differentiable semigroup.

\section{Concluding remarks} \label{6}

%The contribution of the paper is that it established sufficient conditions 
The paper contributed in the construction of mild solution for rarely studied nonlinear fractional evolution systems governed with higher order Riemann-Liouville derivatives involving non-instantaneous impulses. This study guarantees the existence and uniqueness of mild solution of system (\ref{1}) under suitable assumptions. The study of non-instantaneous impulsive systems is broad in scope and therefore, this artefact serves as a tool for further potential study of the concerned system for controllability and stability analysis. The analysis can be performed for approximate or partial controllability with general nonlocal conditions or in inclusions (see \cite{baleanu1,abdul:partial,mahm3,A shukla3} for some idea). Infact, establishment of existence and uniqueness results for mild solutions of the concerned system when Lipschitz continuity is exempted for the nonlinear function is also a matter of prime investigation.

%Non-instantaneous impulsive systems are studied for a variety of reasons.

%%%The article established the results for existence and uniqueness of the non-instantaneous impulsive fractional differential systems involving Riemann-Liouville derivatives. The definition of the mild solution of the system (\ref{1}) has been constructed utilizing the concept of fractional semigroup and appropriate integral type initial conditions. The existence and uniqueness results have been derived with the help of Banach's fixed point approach. 
%%%%Approximate controllability has been established using Lemma 3 along with iterative techniques for control sequence.
%%%The study of non-instantaneous impulsive systems covers a wide range of applications. 
%%%The proposed future work 
%%%%in alignment with the present work
%%% emerges from relaxing the Lipschitz continuity on the nonlinear operator. If the nonlinear operator $h$ is not Lipschitz, then even existence of solution is the matter of prime investigation. Also, the considered system  with general nonlocal conditions can be analyzed for %presented extracts/findings/work
%%%%present findings can be extended for the 
%%%approximate controllability, partial approximate controllability or finite approximate controllability. For some idea, see \cite{A shukla1,abdul:partial,mahm2,mahm3}.

%---- Bibliography ----
%

\end{document}